\documentclass[11pt]{amsart}

\usepackage{amsfonts}
\usepackage{amssymb}
\usepackage{amsmath}
\usepackage{amsthm}
\usepackage{enumerate}
\usepackage{mathrsfs}

\newcommand{\mathscripty}{\mathcal}

\newcommand{\rs}{\mathord{\upharpoonright}}

\newcommand{\e}{\epsilon}

\newcommand{\NN}{\mathbb{N}}

\newcommand{\CC}{\mathbb{C}}

\newcommand{\SB}{\mathscripty{B}}
\newcommand{\SC}{\mathscripty{C}}
\newcommand{\SD}{\mathscripty{D}}

\newcommand{\SF}{\mathscripty{F}}

\newcommand{\SU}{\mathscripty{U}}

\newcommand{\B}{B}

\newcommand{\er}{\mathbb R}

\newtheorem{theorem}{Theorem}[section]
\newtheorem*{theorem*}{Theorem}
\newtheorem{proposition}[theorem]{Proposition}
\newtheorem*{proposition*}{Proposition}
\newtheorem{lemma}[theorem]{Lemma}
\newtheorem*{lemma*}{Lemma}
\newtheorem{corollary}[theorem]{Corollary}
\newtheorem*{corollary*}{Corollar}

\newtheorem*{fact*}{Fact}
\theoremstyle{definition}
\newtheorem{definition}[theorem]{Definition}
\newtheorem*{definition*}{Definition}
\newtheorem{claim}[theorem]{Claim}
\newtheorem*{claim*}{Claim}

\newtheorem*{conjecture*}{Conjecture}

\newtheorem{theoremi}{Theorem}
\newtheorem{corollaryi}[theoremi]{Corollary}

\theoremstyle{remark}

\newtheorem*{example*}{Example}
\newtheorem{remark}[theorem]{Remark}
\newtheorem*{remark*}{Remark}

\newtheorem*{note*}{Note}

\newtheorem*{question*}{Question}

\newcommand{\set}[2]{\left\{#1\mathrel{}\middle|\mathrel{}#2\right\}}

\newcommand{\norm}[1]{\left\lVert #1 \right\rVert}

\DeclareMathOperator{\expn}{exp}

\DeclareMathOperator{\id}{id}

\DeclareMathOperator{\GL}{GL}

\usepackage{enumitem}

\begin{document}

\title{Ulam stability for some classes of C*-algebras}%
\author[P. McKenney]{Paul McKenney}
\address[P. McKenney]{Department of Mathematics \\
Miami University \\
501 E. High Street \\
Oxford, Ohio 45056 USA
}
\email{mckennp2@miamioh.edu}
\urladdr{}

\author[A. Vignati]{Alessandro Vignati}
\address[A. Vignati]{Department of Mathematics and Statistics\\
York University\\
4700 Keele Street\\
Toronto, Ontario\\ Canada, M3J
1P3\\
}
\email{ale.vignati@gmail.com}
\urladdr{www.automorph.net/avignati}


\date{\today}%

\maketitle
\begin{abstract}
We prove some stability results for certain classes of C*-algebras. We prove that whenever $A$ is a finite-dimensional
C*-algebra, $B$ is a C*-algebra and $\phi\colon A\to B$ is approximately a $^*$-homomorphism then there is an actual
$^*$-homomorphism close to $\phi$ by a factor depending only on how far is $\phi$ from being a $^*$-homomorphism and not
on $A$ or $B$.
 \end{abstract}
\section*{Introduction}
In this paper we prove an Ulam stability result in the category of C*-algebras.  Ulam stability results state that a map
which is ``almost'' a morphism must be close to an actual morphism; the exact definition of ``almost-morphisms'', and the
notion of closeness involved, vary with the result.  This area originated and take its name from the work of Ulam, who, in \cite[Chapter VI.1]{Ulam.PR}), was the first to pose stability questions.

While some results on the Ulam stability of group representations can be found in \cite{Kazhdan}, \cite[Section 3]{GKR.Jacobi} and, more recently,
\cite{BurgerOzawaThom.U}, we work instead in the category of C*-algebras, where the notion of an $\e$-morphism
necessarily involves all of the operations of a C*-algebra (see Definition~\ref{def:approximate-homo} below). Throughout the paper $\mathcal F$ will denote the class of finite-dimensional C*-algebras, $\mathcal {AF}$ the class of unital AF algebras, $\mathcal M$ the class of von Neumann algebras and $\mathcal C^*$ the class of all C*-algebras.

Our notion of Ulam stability is quite strong: given two classes $\mathcal C_1$ and $\mathcal C_2$ of C*-algebras, we say
that $(\mathcal{C}_1, \mathcal{C}_2)$ is Ulam stable if for all $\epsilon>0$ there is $\delta>0$ such that given $A\in\mathcal C_1$, $B\in\mathcal C_2$ and
a $\delta$-$^*$-homomorphism $\phi\colon A\to B$, we can find a $^*$-homomorphism
$\psi\colon A\to B$ which is $\e$-close to $\phi$ (see Definition \ref{def:ulamstab}). The same hypotheses were assumed by Farah in~\cite[Theorem 5.1]{Farah.AC}, where he proved a quantitative version of Ulam stability for $\mathcal C_1=\mathcal C_2=\mathcal F$.

In the first part of the paper, we extend Farah's result: Theorem~\ref{thmi:ulam-stability} below gives Ulam stability for $\mathcal{C}_1 =
\mathcal{F}$ and $\mathcal{C}_2 = \mathcal{C}^*$.

\begin{theoremi}\label{thmi:ulam-stability}
For all $\e>0$ there is $\delta>0$ such that given $F\in\mathcal F$, $A\in\mathcal C^*$ and an $\delta$-$^*$-homomorphism  $\phi\colon F\to A$, there exists a $^*$-homomorphism $\psi\colon F\to A$ such
that for all $x\in F_{\le 1}$,
\[
  \norm{\psi(x)-\phi(x)}<\e
\]
\end{theoremi}
(Our Theorem \ref{thm:ulam-stability} gives a quantitative version of Theorem \ref{thmi:ulam-stability}, and the dependence of $\delta$ is terms of $\epsilon$ is specified.)

In the second part of the paper, we prove that Ulam stability passes to inductive limits in case the range
algebra is a von Neumann algebra (see Theorem \ref{thm:limits} for details). As a consequence, we obtain the following:
\begin{corollaryi}
  \label{cor:unitalAF}
  For all $\e>0$ there is $\delta>0$ such that given $A\in\mathcal {AF}$, $M\in\mathcal M$ and an
  $\e$-$^*$-homomorphism $\phi\colon A\to M$, there is a $^*$-homomorphism $\psi\colon A\to M$ such that for all $x\in
  A_{\le 1}$,
  \[
    \norm{\phi(x)-\psi(x)}<\e.
  \]
\end{corollaryi}
(As above, we offer a quantitative version of Corollary \ref{cor:unitalAF} in Corollary \ref{cor:AF}.)

At the end of the paper we give a connection to perturbation theory, showing that a particular form of Ulam stability
for unital separable AF algebras is equivalent to the problem of whether, given two Kadison-Kastler close copies of a
separable AF algebra, there must be a $^*$-isomorphism between them which is uniformly close to the identity (see
Corollary \ref{cor:ulamvsperturb} for details). Phillips and Raeburn proved in~\cite{PhR.AF} that any two
Kadison-Kastler close AF algebras must be isomorphic, but they were not able to control the distance between the
isomorphism and the identity; this problem has been open since.  Later, in \cite{Christensen.NI}, it was proved (among
other things) that every two Kadison-Kastler close copies of an AF algebra are unitarily equivalent. This was extended
to separable nuclear C*-algebras in \cite{CSSWW} (see also \cite[Theorem 2.3]{HKW}), and it is known that separability
(see \cite{ChoiCh.NonSepNotGood}) is necessary.  The question of whether two nonnuclear algebras which are close in the
Kadison-Kastler metric are necessarily isomorphic is still open.

Ulam stability for C*-algebras is, as mentioned, closely connected to results on near inclusion; many results from both areas can be found in~\cite{Johnson.AMNM}, \cite{Johnson.NI}, \cite{Christensen.NI}, \cite{HKW} and \cite{CSSWW}, among others. Our results differ from these in that our notion of an approximate homomorphism does not require linearity; in~\cite{Johnson.AMNM}, Johnson discusses some of the difficulties that arise when the maps involved are not linear.  We also ask that the dependence between $\delta$ and $\epsilon$ is uniform over all algebras and maps involved. 

The motivation for considering such a wide and unnatural class of maps (i.e., nonlinear maps) is given by the study of automorphisms of corona
algebras (see \cite{Farah.AC}, \cite{Ghasemi.FDD}, \cite{McKenney.UHF} and the upcoming \cite{MKAV.AC}). In general,
Ulam stability results find applications in the theory of rigidity of quotients, where the goal is to find, under some
additional set theoretical assumption, some well behaved lifting for morphisms between quotient structures. Examples in a
discrete setting can be found in \cite{Farah.LIFU} or \cite{KanoveiReeken.BR}, while in the continuous setting
\cite[Theorem 5.1]{Farah.AC} was crucial in determining that under the Open Coloring Axiom all automorphisms of the
Calkin algebra are inner.

We would like to point out some obstructions that prevent us from extending Theorem \ref{thmi:ulam-stability} and
Corollary \ref{cor:unitalAF}. In the proof of Theorem \ref{thmi:ulam-stability}, and in particular in the
application of Proposition \ref{proposition:approximate-multiplicativity}, we make heavy use of the compactness of the unitary
group of a finite-dimensional C*-algebra; in particular, we take advantage of the Haar measure several times to perform
``averaging techniques'' that remove irregularities in the given $\e$-$^*$-homomorphism. If the group is not
compact (as is the unitary group of every infinite dimensional C*-algebra), such techniques fail. A more
specific explanation of the difficulties in obtaining stability results for noncompact groups may be found in
\cite{BurgerOzawaThom.U}.

As for Corollary~\ref{cor:unitalAF}, the necessity of having a weak-$*$-closed range was already noted in \cite[Theorem
3.1]{Johnson.AMNM}, where the range varies among dual Banach algebras. Similarly, in \cite[Section 4]{Christensen.NI},
having a von Neumann algebra in the range is crucial. For near inclusion phenomena (a particular case of Ulam stability) in
the absence of a weak-$*$-closed range, the sharpest result that has been obtained so far is \cite[Theorem 2.3]{HKW}.

We would like to thank Caleb Eckhardt, George Elliott, Ilijas Farah and Stuart White for the countless remarks and
suggestions.  In particular, we would like to thank George Elliott for suggesting the use of the Peter-Weyl theorem in
the proof of our main result, and Stuart White for suggesting the statement of Corollary \ref{cor:unitalAF}.

\section{The main result}
First we must introduce the precise definition of $\e$-$^*$-homomorphism that we will be using.
\begin{definition}
  \label{def:approximate-homo}
  A map $\phi \colon A\to B$ between C*-algebras is called an \emph{$\e$-$^*$-homomorphism} if for all $x,y\in A_{\le 1}$ and
  $\lambda\in \CC_{\le 1}$, we have
  \begin{gather*}
    \norm{\phi(x + y) - \phi(x) - \phi(y)} \le \e, \\
    \norm{\phi(\lambda x) - \lambda \phi(x)} \le \e, \\
    \norm{\phi(xy) - \phi(x)\phi(y)} \le \e, \\
    \norm{\phi(x^*) - \phi(x)^*} \le \e, \\
    \norm{\phi(x)} \le 1 + \e
\end{gather*}
\end{definition}
The notion of closeness that we will use is simply the metric induced by the uniform norm over the unit ball, which of
course coincides with the operator norm whenever the maps in consideration are linear:
\begin{definition}
  If $\phi : A\to B$ is a map between C*-algebras, then we will write $\norm{\phi}$ for the quantity
  \[
    \sup\set{\norm{\phi(x)}}{x \in A_{\le 1}}
  \]

A map $\phi\colon A\to B$ is called \emph{$\e$-isometric} if for all $x$ with $\norm{x}=1$ we have
\[
  \norm{\phi(x)}\in [1-\e,1+\e]
\]
and $\phi$ is said to be \emph{$\e$-surjective} if for all $b\in B_{\le 1}$ there is $a\in A$ with 
\begin{gather*}
  \norm{\phi(a) - b}\le \e.
\end{gather*}
We define an $\e$-$^*$-isomorphism to be an $\e$-isometric, $\e$-surjective $\e$-$^*$-homomorphism.
%Relaxing the definition of being nonzero,
We say that a map $\phi$ is \emph{$\e$-nonzero} if there is $a\in A$ with $\norm{a}=1$ and $\norm{\phi(a)}\geq 1-\e$.
\end{definition}
\begin{remark}
\label{rmk:unital-contractive}
To aid our calculations later on, we will often assume that $\norm{\phi} \le 1$.  For our results, this gives no loss of
generality, since if $\phi$ is an $\e$-$^*$-homomorphism as defined above, and $\norm{\phi} > 1$, then $\psi =
\frac{1}{\norm{\phi}}\phi$ satisfies $\norm{\phi - \psi} \le \e$.  Similarly, if $A$ is unital and $\e$ is small enough,
then we may assume without loss of generality that $\phi(1)$ is a projection.  To see this, note that $\phi(1)$ is an almost-projection
and hence (by standard spectral theory tricks) is close to an actual projection $p\in B$.  Then by replacing $\phi(1)$
with $p$, we get a unital $\delta$-$^*$-homomorphism, where $\delta$ is polynomial in $\e$.
\end{remark}

%Similarly, if $A$ is unital and $\e$ is small enough, we can (and will) assume without loss of generality that $B$ is unital, and that $\phi$
%is unital as well.  To see this, first note that we can replace $\phi$ with the map $\psi(x) = that $\norm{\phi(1)^2 - \phi(1)} \le \e$ and $\norm{\phi(1^*) - \phi(1)^*}
%\le \e$.  Using standard tricks, we get a projection $p\in B$ with $\norm{\phi(1) - p} \le 2\e$.  Then the map $\psi(x)
%= p\phi(x)p$ 
%considering that $\phi(1)$ is an almost projection in $B$, hence it is close, up to $2\e$, to a
%projection $p\in B$ and we can then consider $p\phi(\cdot)p$ in place of $\phi$. Moreover, if $A$ is not unital, it is
%always possible to extend $\phi$ to an $\e$-$^*$-homomorphism from unitization of $A$ to the unitization of $B$.  For
%these reasons we will only consider unital $\e$-$^*$-homomorphisms.
%\end{remark}

It should be noted that the definition of an $\e$-$^*$-homomorphism provided in \cite{Farah.AC} was in fact our definition
of an $\e$-isometric $\e$-$^*$-homomorphism.  When $A$ is a full matrix algebra, and $\e$ is sufficiently small, being an $\e$-isometry is automatic:
\begin{proposition}
Suppose $\e < \frac{1}{100}$, $\ell\in\NN$, $B$ is a C*-algebra, and $\phi : M_\ell\to B$ is a $2\sqrt{\e}$-nonzero
$\e$-$^*$-homomorphism with $\norm{\phi} \le 1$.  Then $\phi$ is $2\sqrt{\e}$-isometric.
\end{proposition}
\begin{proof}
Suppose that there is $x$ of norm $1$ with $\norm{\phi(x)}\le 1-2\sqrt{\e}$.  Note that for any $a\in A_{\leq 1}$,
\[
  \left|\norm{\phi(a^*a)} - \norm{\phi(a)}^2\right| \le 2\e
\]
(Here we are using the fact that $\norm{\phi} \le 1$.)  Let $s(a) = a^* a$ for all $a\in A$.
\begin{claim}
  There is an $n\in\NN$ such that $\norm{\phi(s^{(n)}(x))}\le 2\sqrt{\e}$.
\end{claim}
\begin{proof}
Let $k\in\NN$.  Observe that
\begin{equation}
  \label{eq-induction} (1 - k\sqrt{\e})^2 + 2\e \le 1 - (k+1)\sqrt{\e}
\end{equation}
if and only if
\[
  k^2 - \frac{1}{\sqrt{\e}} k + \left(2 + \frac{1}{\sqrt{\e}}\right) \le 0
\]
if and only if
\[
  \frac{1}{2\sqrt{\e}}(1 - \tau) \le k \le \frac{1}{2\sqrt{\e}}(1 + \tau)
\]
where
\[
  \tau = \sqrt{1 - 4\left(2\e + \sqrt{\e}\right)}.
\]
By Taylor's theorem, and our assumption that $\e < \frac{1}{100}$, we have
\[
  \tau \ge 1 - 2(2\e + \sqrt{\e}) - (2\e + \sqrt{\e})^2 \ge 1 - 4\sqrt{\e}.
\]
It follows that the inequality~\eqref{eq-induction} holds for all positive integers $k$ in the range
\[
  2 \le k \le \frac{1}{\sqrt{\e}} - 2.
\]
Since $x$ is such that $\norm{\phi(x)}\leq 1-2\sqrt\e$, we have that 
\[
\norm{\phi(s(x))}\leq\norm{\phi(x)}^2+2\e\leq 1-4\sqrt\e+4\e\leq 1-2\sqrt\e.
\]
By repeatedly applying $s$, for $k\leq\frac{1}{\sqrt\e}-2$, we get that 
\[
  \norm{\phi(s^{(k)}(x))} \le 1- (k+1)\sqrt{\e}.
\]
In particular, if $n$ is the maximal integer smaller than $\frac{1}{\sqrt{\e}} - 2$, then
\[
  \norm{\phi(s^{(n)}(x))} \le 1 - \left(\frac{1}{\sqrt{\e}} - 2\right)\sqrt{\e} = 2\sqrt{\e},
\]
as required.
\end{proof}

Replacing $x$ with $s^n(x)$, we can assume that $x$ is positive, $\norm{\phi(x)}\leq 2\sqrt{\e}$ and $1\in\sigma(x)$. In particular, there is a projection $p\in M_\ell$ of rank $1$ such that
$pxp=p$.  Then,
\[
  \norm{\phi(p)} = \norm{\phi(pxp)}\le \norm{\phi(p)}^2\norm{\phi(x)}+2\e\le 2\sqrt{\e} + 2\e,
\]
which implies $\norm{\phi(p)}\leq\frac{1}{4}$. For the same reason as before, since $s^n(p)=p$ for all $n$, we have $\norm{\phi(p)}\leq 2\sqrt\e$. As every two projections of the same rank in $M_\ell$ are unitarily
equivalent, every projection of rank $1$ has image of small norm. Note that for every projection $p\in M_l$, either $\norm{\phi(p)}\leq2\sqrt\e$ or $\norm{\phi(p)}\geq \frac{1}{2}$.

Let $j\leq \ell$ be the minimum such that there is a projection $p$ of rank $j$ with $\norm{\phi(p)}\ge\frac{1}{2}$.
Since $\norm{\phi(1)}\ge \frac{1}{2}$, $j$ exists and, by the above, $j>1$. Let $q_1,q_2$ be projections of rank smaller
than $j$, with $p=q_1+q_2$. We have
\[
  \norm{\phi(p)}\leq\norm{\phi(q_1)+\phi(q_2)}+\e\leq\norm{\phi(q_1)}+\norm{\phi(q_2)}+\e<4\sqrt\e+\e<\frac{1}{2}
\] 
a contradiction to $\norm{\phi(p)}\ge\frac{1}{2}$.
\end{proof}
We now can give the definition of stability we are going to use throughout the paper.
\begin{definition}\label{def:ulamstab}
 Let $\SC$ and $\SD$ be two classes of C*-algebras.  We say that the pair $(\SC,\SD)$ is \emph{Ulam stable} if for every
 $\e > 0$ there is a $\delta > 0$ such that for all $A\in\SC$ and $B\in\SD$ and for every $\delta$-$^*$-homomorphism
 $\phi : A\to B$, there is a $^*$-homomorphism $\psi : A\to B$ with $\norm{\phi - \psi} < \e$.

 %More specifically, we say that $(\SC,\SD)$ is Ulam stable \emph{with exponent $r$} if there are constants $\gamma, K > 0$ such that for all $\e < \gamma$, $A\in\SC$ and $B\in\SD$, and for every $\e$-$^*$-homomorphism $\phi : A\to B$, there is a $^*$-homomorphism $\psi : A\to B$ with $\norm{\phi - \psi} < K\e^r$.
\end{definition}
Recall that $\mathcal F$ denotes the class of finite-dimensional C*-algebras, and $\mathcal C^*$ the class of all C*-algebras. The following is Theorem \ref{thmi:ulam-stability} above:
\begin{theorem}\label{thm:ulam-stability}
There are $K,\delta>0$ such that given $\e<\delta$, $F\in\mathcal F$, $A\in\mathcal C^*$ and an $\e$-$^*$-homomorphism $\phi\colon F\to A$, there exists a $^*$-homomorphism $\psi\colon F\to A$ with \[\norm{\psi-\phi}<K\e^{1/2}.\]
Consequently, the pair $(\mathcal F,\mathcal C^*)$ is Ulam stable.
\end{theorem}
The proof goes through successive approximations of an $\e$-$^*$-homomorphism $\phi$ with increasingly nice properties. Each step will consist of an already-known approximation result; our proof will thus consist of stringing each of these results together, sometimes with a little work in between.  Before beginning the proof we describe some of the tools we will use.

%The following lemma is already well-known.
%\begin{lemma}\label{lemma:approximate-projection}
  %If $0 < \e < \frac{1}{8}$, $A\in\mathcal C^*$ and $a\in A$ satisfies
  %\[
    %\norm{a^2 - a} < \e\qquad \norm{a^* - a} < \e
  %\]
 % then there is a projection $p\in A$ such that $\norm{a - p} < 4\e$.
%\end{lemma}
The following Proposition is essentially proved in~\cite[Proposition~5.14]{AGG}; one can also find similar ideas in the proof of~\cite[Proposition~5.2]{Kazhdan}.  Our version is slightly more general,
in that the values of $\rho$ are taken from the invertible elements of a separable Banach algebra, and $\rho$ is allowed to be just Borel measurable.  In our proof, we will need the Bochner integral,
which is defined for certain functions taking values in a Banach space.  For an introduction to the Bochner integral and its properties, we refer the reader to~\cite[Appendix~E]{Cohn}.  For our
purposes, we note that the Bochner integral is defined for any measurable function $f$ from a measure space $(X,\Sigma,\mu)$ into a separable Banach space $E$ such that the function $x\mapsto
\norm{f(x)}$ is in $L^1(X,\Sigma,\mu)$, and in this case,
\[
  \int f(x)\,d\mu(x) \in E
\]
and
\[
  \norm{\int f(x)\,d\mu(x)} \le \int \norm{f(x)}\,d\mu(x).
\]
Moreover, if $G$ is a compact group and
$\mu$ is the Haar measure on $G$, then for any Bochner-integrable $f : G\to E$ and $g\in G$ we have
\[
  \int f(x)\,d\mu(x) = \int f(gx)\,d\mu(x).
\]
\begin{proposition}
  \label{proposition:approximate-multiplicativity}
  Suppose $A$ is a separable Banach algebra, $G$ is a compact group, and $\rho \colon G\to \GL(A)$ is a Borel-measurable map satisfying, for all $u,v\in G$,
  \[
    \norm{\rho(u)^{-1}} \le \kappa
  \]
  and
  \[
    \norm{\rho(uv) - \rho(u)\rho(v)} \le \e
  \]
  where $\kappa$ and $\e$ are positive constants satisfying $\e < \kappa^{-2}$.  Then there is a Borel-measurable $\tilde{\rho} \colon G\to\GL(A)$ such that
  \begin{enumerate}
    \item\label{rho:close} for all $u\in G$, $\norm{\tilde{\rho}(u) - \rho(u)} \le \kappa\e$,
    \item\label{rho:invertible} for all $u\in G$,
    \[
      \norm{\tilde{\rho}(u)^{-1}} \le \frac{\kappa}{1 - \kappa^2\e},
    \]
    and finally,
    \item\label{rho:multiplicative} for all $u,v\in G$,
    \[
      \norm{\tilde{\rho}(uv) - \tilde{\rho}(u)\tilde{\rho}(v)} \le 2\kappa^2\e^2.
    \]
  \end{enumerate}
\end{proposition}

\begin{proof}
  Define
  \[
    \tilde{\rho}(u) = \int \rho(x)^{-1}\rho(xu)\,d\mu(x)
  \]
  where $\mu$ is the Haar measure on $G$, and the integral above is the Bochner integral.  Clearly, $\tilde{\rho}$ is Borel-measurable.  To check condition~\eqref{rho:close}, we have
  \begin{align*}
    \norm{\tilde{\rho}(u) - \rho(u)} & \le \int \norm{\rho(x)^{-1}\rho(xu) - \rho(u)}\,d\mu(x) \\
      & \le \int \norm{\rho(x)^{-1}}\norm{\rho(xu) - \rho(x)\rho(u)}\,d\mu(x) \le \kappa\e.
  \end{align*}
  Note now that
  \[
    \norm{1 - \tilde{\rho}(u)\rho(u)^{-1}} \le \norm{\rho(u) - \tilde{\rho}(u)}\norm{\rho(u)^{-1}} \le \kappa^2\e.
  \]
  By standard spectral theory, since $\rho(u)$ is invertible and $\norm{\tilde{\rho}(u) - \rho(u)} < 1$, we have that $\tilde{\rho}(u)$ is invertible too, and moreover
  \begin{align*}
    \norm{\tilde{\rho}(u)^{-1}} & \le \norm{\rho(u)^{-1}}\left(1 + \norm{1 - \tilde{\rho}(u)\rho(u)^{-1}} + \norm{1 - \tilde{\rho}(u)\rho(u)^{-1}}^2 + \cdots\right) \\
      & \le \frac{\kappa}{1 - \kappa^2\e}
  \end{align*}
  which proves condition~\eqref{rho:invertible}.  The real work comes now in proving condition~\eqref{rho:multiplicative}.  First, we note that
  \begin{eqnarray*}
    \tilde{\rho}(u)\tilde{\rho}(v) - \tilde{\rho}(uv) &=& \iint \left(\rho(x)^{-1}\rho(xu) \rho(y)^{-1} \rho(yv) - \rho(x)^{-1}\rho(xuv)\right)\,d\mu(x)\,d\mu(y)\\& =& I_1 + I_2
  \end{eqnarray*}
  where
  \[
    I_1 = \iint \left(\rho(x)^{-1}\rho(xu) - \rho(u)\right)\left(\rho(y)^{-1}\rho(yv) - \rho(v)\right)\,d\mu(x)\,d\mu(y)
  \]
  and
  \[
    I_2 = \iint \left(\rho(x)^{-1}\rho(xu)\rho(v) + \rho(u)\rho(y)^{-1}\rho(yv) - \rho(u)\rho(v) - \rho(x)^{-1}\rho(xuv)\right)\,d\mu(x)\,d\mu(y).
  \]
  For $I_1$ we have
  \[
    \norm{I_1} \le \iint \norm{\rho(x)^{-1}}\norm{\rho(xu) - \rho(x)\rho(u)}\norm{\rho(y)^{-1}}\norm{\rho(yv) - \rho(y)\rho(v)}\,d\mu(x)\,d\mu(y) \le \kappa^2\e^2.
  \]
  As for $I_2$, we have
  \[
    I_2 = \int \rho(x)^{-1} (\rho(xu)\rho(v) - \rho(xuv))\,d\mu(x) - \int (\rho(u)\rho(x)^{-1}\rho(x)\rho(v) - \rho(u)\rho(x)^{-1}\rho(xv))\,d\mu(x).
  \]
  Using the translation-invariance of $\mu$ on the first integral above to replace $xu$ with $x$, we see that
  \begin{align*}
    I_2 & = \int \rho(xu^{-1})^{-1} (\rho(x)\rho(v) - \rho(xv))\,d\mu(x) - \int \rho(u)\rho(x)^{-1}(\rho(x)\rho(v) - \rho(xv))\,d\mu(x) \\
      & = \int (\rho(xu^{-1})^{-1} - \rho(u)\rho(x)^{-1})(\rho(x)\rho(v) - \rho(xv))\,d\mu(x)
  \end{align*}
  Finally, note that
  \[
    \norm{\rho(xu^{-1})^{-1} - \rho(u)\rho(x)^{-1}} = \norm{\rho(xu^{-1})^{-1}(\rho(x) - \rho(xu^{-1})\rho(u))\rho(x)^{-1}} \le \kappa^2\e
  \]
  and
  \[
    \norm{\rho(x)\rho(v) - \rho(xv)} \le \e
  \]
  so we have that $\norm{I_2} \le \kappa^2\e^2$.  This proves condition~\eqref{rho:multiplicative}.
\end{proof}
Lastly, for convenience, we state Farah's result:
\begin{theorem}[Theorem 5.1, \cite{Farah.C}]\label{thm:farahulam}
There are constants $K_1,\gamma> 0$ such that whenever $\e <\gamma$, $F_1,F_2\in\SF$ and $\phi \colon F_1\to F_2$ is an $\e$-$^*$-homomorphism, there is a $^*$-homomorphism $\psi \colon F_1\to F_2$ with $\norm{\phi - \psi} < K_1 \e$.
  Hence, the pair $(\SF,\SF)$ is Ulam stable.
\end{theorem}
We are now ready to prove our main result.  In the proof we will make several successive modifications to $\phi$, and in
each case the relevant $\e$ will increase by some linear factor.  In order to keep the notation readable, we will call
the resulting $\e$'s $\e_1, \e_2, \ldots$
\begin{proof}[Proof of Theorem~\ref{thm:ulam-stability}]
Let $\gamma, K > 0$ witness Farah's Theorem.  Let $\delta \ll \gamma, 1/K$.  We will in particular require $\delta <
2^{-12}$.
%\[
%  \delta = \min\left\{2^{-10},\frac{\gamma}{100},\frac{1}{10^7 K}\right\}
%\]
%Fix $\delta$ small enough that
%\begin{itemize}
%  \item  for every $\e<\delta$ we have
%  \[
%    \sum_{n\geq 0} \e_n \le 2\e
%  \]
%  where $\e_0 = \e$ and $\e_{n+1} = 2\e_n^2$.
%  \item  $\delta < \gamma / 100$, and
%  \item  $\delta < 10^{-7} K^{-1}$.
%\end{itemize}
%(These approximations are not optimal, but sufficient.)
Fix $\e < \delta$, $A\in\mathcal C^*$, $F\in\mathcal F$, and an
$\e$-$^*$-homomorphism $\phi \colon F \to A$.   As in Remark~\ref{rmk:unital-contractive}, we will assume that $A$ is
unital, $\phi(1) = 1$, and $\norm{\phi} \le 1$.

Let $X = \{x_0,\ldots,x_k\}$ be a finite subset of $F_{\le 2}$ which is $\e$-dense in $F_{\le 2}$ and which includes
$1$.  Define a map $\phi' \colon F_{\le 2} \to A$ by letting $\phi'(x) = \phi(x_i)$, where $i$ is the minimal integer
such that $\norm{x - x_i} < \e$.  Clearly, the range of $\phi'$ is just $\{\phi(x_0),\ldots,\phi(x_k)\}$, and if $B_i =
B(x_i,\e)\cap F_{\le 2}$, then
\[
  (\phi')^{-1}(\phi(x_i)) = B_i \setminus \bigcup_{j < i} B_j
\]
so $\phi'$ is a Borel map.  Moreover, $\norm{\phi'(x) - \phi(x)} < \e$ for all $x\in F_{\le 2}$.  It follows that
$\phi'$ is an $\e_1$-$^*$-homomorphism, where $\e_1 = 4\e$.  Note also that $\phi'(1) = 1$ and $\norm{\phi'} \le 1$.  Replacing $\phi$ with
$\phi'$ and $A$ with the C*-algebra generated by $\{\phi(x_0),\ldots,\phi(x_k)\}$, we may assume that $\phi$ is
Borel-measurable and $A$ is separable (at the expense of restricting the domain of $\phi$ to $F_{\le 2}$).

Since $\e_1 < 1$ and $\phi$ is unital, it follows that for every $u\in\SU(F)$, we have $\norm{\phi(u^{-1})\phi(u) -
1} < 1$ and hence that $\phi(u)$ is invertible, and $\norm{\phi(u)^{-1}} \le 2$.  Let $\rho_0$ be the restriction of $\phi$ to $\SU(F)$.   Applying
Proposition~\ref{proposition:approximate-multiplicativity} repeatedly, we may find a sequence of maps $\rho_n :
\SU(F)\to \GL(A)$ satisfying, for all $u,v\in \SU(F)$, 
\[
  \norm{\rho_n(uv) - \rho_n(u)\rho_n(v)} \le \delta_n \qquad \norm{\rho_{n+1}(u) - \rho_n(u)} \le \kappa_n\delta_n \qquad \norm{\rho_n(u)^{-1}} \le \kappa_n
\]
where $\delta_n$ and $\kappa_n$ are defined by letting $\delta_0 = \e_1$, $\kappa_0 = 2$, and
\[
  \delta_{n+1} = 2\kappa_n^2\delta_n^2 \qquad \qquad \kappa_{n+1} = \frac{\kappa_n}{1 - \kappa_n^2 \delta_n}
\]
\begin{claim}
  For each $n$, $\kappa_{n+1} - \kappa_n < 2^{-n}$ and $\delta_n \le 2^{5(1 - 2^n)}\e_1$.  Consequently, $\kappa_n < 4$ for all
  $n$, and
  \[
    \sum_{n=0}^\infty \kappa_n \delta_n < 8\e_1.
  \]
\end{claim}
\begin{proof}
  We will prove that $\kappa_{n+1} - \kappa_n < 2^{-n}$ and $\delta_n \le 2^{5(1 - 2^n)}\e_1$ by induction on $n$.  For the base
  case we note that $\delta_0 = \e_1 = 4\e < 2^{-10}$,
  \[
    \kappa_1 - \kappa_0 \le \frac{2}{1 - 2^{-8}} - 2 < 1.
  \]
  Now suppose $\kappa_0,\ldots,\kappa_n$ and $\delta_0,\ldots,\delta_n$ satisfy the induction hypothesis above.  Then we clearly
  have
  \[
    \kappa_n < 2 + 1 + \cdots + 2^{-(n-1)} < 4.
  \]
  Then using this fact and the assumption $\e_1 < 2^{-10}$,
  \[
    \delta_{n+1} = 2\kappa_n^2 \delta_n^2
      < 2(4^2) 2^{10(1 - 2^n)}\e_1^2
      < 2(4^2) (2^{-10}) 2^{10(1 - 2^n)}\e_1
      = 2^{5(1 - 2^{n+1})}\e_1.
  \]
  Moreover,
  \[
    \kappa_{n+1} - \kappa_n = \frac{\kappa_n^3 \delta_n}{1 - \kappa_n^2\delta_n} < \frac{(4^3) 2^{5(1 - 2^n)}\e_1}{1 -
    (4^2) 2^{5(1 - 2^n)}\e_1} < \frac{2^{1 - 5\cdot 2^n}}{1 - 2^{-1}} = 2^{2 - 5\cdot 2^n}
  \]
  Finally, note that $2 - 5\cdot 2^n \le -n$ for all $n\ge 0$.  This proves the first two parts of the claim.  We have
  already noted that $\kappa_n < 2 + 1 + \cdots + 2^{-(n-1)} < 4$.  As for the other sum, we have
  \[
    \sum_{n=0}^\infty \kappa_n \delta_n < 4\e_1 \sum_{n=0}^\infty 2^{5(1 - 2^n)} < 4\e_1 \sum_{n=0}^\infty 2^{-n} =
    8\e_1.
  \]
\end{proof}

It follows from the above claim that the map $\rho$ given by $\rho(u) = \lim \rho_n(u)$ is defined on $\SU(F)$, maps
into $\GL(A)$, and is multiplicative, Borel-measurable, and satisfies $\norm{\rho - \phi} \le  8\e_1 = \e_2$.

Fix a faithful representation $\sigma$ of $A$ on a separable Hilbert space $H$, and let $\tau \colon \SU(F)\to\GL(H)$ be
the composition $\sigma\circ\rho$.  Then $\tau$ is a group homomorphism which is Borel-measurable with respect to the
strong operator topology on $\B(H)$.  Moreover, $\norm{\tau(u)} \le 1 + \e_2$ and $\norm{\tau(u)^*\tau(u) - 1} \le
\e_2(4 + \e_2) = \e_3$
for all $u\in\SU(F)$.  Since $\SU(F)$ is compact, and hence unitarizable, it follows that there is a $T\in\GL(H)$ such
that $\pi(u) = T \tau(u) T^{-1}$ is unitary for every $u\in\SU(F)$, and moreover the proof in this case shows that we
may choose $T$ so that $\norm{T - 1} \le \e_3$.  It follows that
\[
  \norm{\pi(u) - \tau(u)} \le \frac{2(1 + \e_2)\e_3}{1 - \e_3} = \e_4.
\]

Recall that $\SU(F)$, with the norm topology, and $\SU(H)$, with the strong operator topology, are Polish groups; then,
by Pettis's Theorem (see e.g.,~\cite[Theorem~2.2]{Rosendal.AC}), it follows that $\pi$, a Borel-measurable group
homomorphism, is continuous with respect to these topologies.  By the Peter-Weyl Theorem, we may write $H = \bigoplus
H_k$, where each $H_k$ is finite-dimensional and $\pi\rs H_k$ is irreducible. In particular, if $p_k=\text{proj}(H_k)$,
we have that for every $k\in\NN$ and $u\in \SU(F)$, $[p_k,\pi(u)]=0$, and moreover $\pi(u)=\sum p_k\pi(u)p_k$.  Now,
recall that $\norm{\phi(u) - \rho(u)} \le \e_2$ for each $u\in\SU(F)$; hence
\[
  \norm{\sigma(\phi(u)) - \pi(u)} \le \norm{\sigma(\phi(u)) - \tau(u)} + \norm{\tau(u) - \pi(u)} \le \e_4 + \e_2.
\]
It follows that $\norm{[\sigma(\phi(u)), p_k]} \le 2(\e_4 + \e_2)$ for each $u\in\SU(F)$ and $k\in\NN$.  Since each element $a$
of a unital C*-algebra is a linear combination of 4 unitaries whose coefficients have absolute value at most $\norm{a}$,
we deduce that
\[
  \sup_{a\in F,\norm{a}\leq 1}\norm{[\sigma(\phi(a)),p_k]}\leq 8(\e_4 + \e_2) + 8\e_1
\]
Let $\phi_k$ be defined as 
\[
  \phi_k(a)=p_k((\sigma\circ \phi)(a))p_k.
\]
It is not hard to show that $\phi_k\colon F\to\SB(H_k)$ is an $\e_5$-$^*$-homomorphism, where $\e_5 = 8(\e_4 + \e_2) +
9\e_1$.  (In fact, $\phi_k$ is nearly an $\e_1$-$^*$-homomorphism; however, to check that $\phi_k(ab) -
\phi_k(a)\phi_k(b)$ is small we need the norm on the commutator computed above.)  By \cite[Theorem 5.1]{Farah.C} and our
choice of $\gamma$ and $K$, there is a $^*$-homomorphism $\psi_k\colon F\to\mathcal B(H_k)$ such that
$\norm{\phi_k-\psi_k} \le K\e_5$.

Consider now $\psi'=\bigoplus\psi_k$ and the C*-algebras $C = \psi'[F]$ and $B = \sigma[A]$. For every $u\in\mathcal
U(F)$, we have
\[
  \norm{\psi'(u)-\pi(u)}=\sup_k\norm{\psi_k(u)-p_k\pi(u)p_k} \le K\e_5 + \e_4 + \e_2.
\]
Since we also have $\norm{\pi(u) - \sigma(\phi(u))} \le \e_4 + \e_2$, it follows that $C \subset^{\e_6} B$, where $\e_6
= K\e_5 + 2\e_4 + 2\e_2$.
By~\cite[Theorem~5.3]{Christensen.NI}, there is a partial isometry $V\in\B(H)$ such that $\norm{V-1_H}<120\e_6^{1/2}$ and
$V C V^* \subseteq B$.  In particular, $V$ is unitary, and the $^*$-homomorphism $\eta \colon F\to B$ defined by
$\eta(a) = V \psi'(a) V^*$ satisfies $\norm{\eta(a) - \psi'(a)} < 240\e_6^{1/2}$.  Since $\sigma$ is injective, for every $x\in F$ we can define
\[
  \psi(x)=\sigma^{-1}(\eta(x)).
\]
Then $\psi$ is a $*$-homomomorphism mapping $F$ into $A$.
Moreover by construction we have that
\[
  \norm{\psi-\phi}<L\e^{1/2},
\]
where $L$ is a constant independent of $\e$, the dimension of $F$, $A$, and $\phi$.  This completes the proof.
\end{proof}
\subsubsection*{An alternative proof}
We would like to remark that it is possible to obtain the conclusion of Theorem~\ref{thm:ulam-stability} via a slightly different argument. This different approach gives us the occasion of remarking the following proposition, that is essentially contained in the proof of \cite[Theorem 5.1]{Farah.C}, but contains some small mistakes that we would like to correct for future reference. We therefore provide a sketch of the proof.
\begin{lemma}(Farah)
  \label{thm:unitary-lift}
  Let $\e > 0$ be sufficiently small, and let $F,G\in\mathcal F$.  Let $\pi
  \colon \SU(F)\to \SU(G)$ be a continuous group homomorphism such that there is some
  $\e$-$^*$-homomorphism $\phi \colon F\to G$ with $\norm{\phi(u) - \pi(u)}\le \e$ for all
  $u\in\SU(F)$.  Then $\pi$ extends to a $^*$-homomorphism from $F$ to $G$.
\end{lemma}
The proof can be taken verbatim from \cite[p. 22]{Farah.C}, as long as a few modifications are provided to make Farah's argument correct. Let us assume that $F,G$ are matrix algebras, and consider a self-adjoint unitary $a\in F$. Then, applying Stone's Theorem to the 1-parameter group $\pi(\exp^{ira})$, for $r\in\er$, to find a self-adjoint unitary $\rho(a)\in G$ such that $\pi(\expn(ira))=\expn(ir\rho(a))$, for $r\in\er$. (Note that in Farah's paper $a$ is not required to be self-adjoint, and therefore $\rho(a)$ might not be unique). It is still possible to show that $\rho(1_F)=1_G$ and, given a projection $p$ and the corresponding self-adjoint unitary $u=1-2p$, it is possible to define \[\rho(p)=\frac{1-\rho(u)}{2}.\] The proof from this point goes as in Farah's Theorem, eliminating the part where is stated that $\expn(iru)=\expn^{ir}u$, as this slight incorrectness is not necessary for the success of the proof.

This lemma, interesting on its own right, leads to a different proof of Theorem \ref{thm:ulam-stability}: take a faithful representation $\sigma\colon A\to\mathcal B(H)$, and let $\pi\colon\mathcal U(F)\to\mathcal B(H)$ be a a continuous group homomorphism approximating $\sigma\circ f$ on the unitary group as before. Letting $H=\bigoplus H_k$, with $H_k$ finite-dimensional and $\pi\rs H_k$ irreducible the decomposition provided by Peter-Weyl's Theorem, we have that for every $k$, $\pi\rs H_k$ respects the hypothesis of Lemma \ref{thm:unitary-lift}. Hence, for every $k$, there is a $^*$-homomorphism $\psi_k \colon F\to \B(H_k)$ which extends $\pi\rs H_k$. Defining $\psi' = \bigoplus \psi_k$ and $C=\psi'[F]$ we can proceed as in the original proof.

\section{Further results: inductive limits and von Neumann algebra}

Given a class of unital nuclear C*-algebras $\mathcal C$, let $\mathcal D_\mathcal C$ be the class of all unital inductive limits of C*-algebras in $\mathcal C$.
In formulas, $A\in\mathcal D_\mathcal C$ if and only if there are $A_\lambda\in\mathcal C$, for $\lambda\in\Lambda$, with
\begin{itemize}
\item $A_\lambda\subseteq A_\mu$ for every $\lambda<\mu\in\Lambda$, where the inclusion is unital;
\item $\overline{\bigcup_{\lambda\in\Lambda}}A_\lambda=A$.
\end{itemize}
(for example, if $\mathcal C$ denotes the class of full matrix algebras, $\mathcal D_\mathcal C$ is the class of all separable UHF algebras. If $\mathcal C=\mathcal F$, then $\mathcal D_\mathcal C=\mathcal {AF}$.)

The goal of this section is to prove the following Theorem:
\begin{theorem}\label{thm:limits}
Let $\mathcal C$ be a class of unital nuclear C*-algebras. If $(\mathcal C,\mathcal M)$ is Ulam stable, so is $(\mathcal D_{\mathcal C},\mathcal M)$.
\end{theorem}
The first part of the following corollary is proved by combining Theorem \ref{thm:ulam-stability} and Theorem
\ref{thm:limits}.  The second part can be seen from the proof of Theorem \ref{thm:limits}.
\begin{corollary}\label{cor:AF}
$(\mathcal {AF},\mathcal M)$ is Ulam stable. Moreover there is $K$ such that whenever $A\in\mathcal{AF}$,
$M\in\mathcal{M}$, $\e>0$, and $\phi : A\to M$ is an $\e$-$^*$-homomorphism, there is a $^*$-homomorphism $\psi\colon
A\to M$ with $\norm{\phi-\psi}<K\e^{1/4}$.
\end{corollary}
It should be pointed out that we do not require, in the statement of Theorem \ref{thm:limits}, the $\e$-$^*$-homomorphisms
to be $\delta$-isometric, for any $\delta$.

We will make use of a small proposition and of a consequence of \cite[Theorem 7.2]{Johnson.AMNM}:
\begin{proposition}\label{lem:closeness}
Let $M$ be a von Neumann algebra and $x\in M$, $Y\subseteq M$ such that $\norm{x - y} \le \e$ for all $y\in Y$.  If
$z$ is any WOT-accumulation point of $Y$, then $\norm{x - z} \le \e$.
\end{proposition}
\begin{theorem}\label{thm:johnrevised}
There is $K$ such that for any unital, nuclear C*-algebra $A$, von Neumann algebra $M$,
$\e>0$ and for any linear $\e$-$^*$-homomorphism $\phi\colon A\to M$, there is a $^*$-homomorphism $\psi$ with
$\norm{\phi-\psi}<K\e^{\frac{1}{2}}$. 
\end{theorem}

\begin{remark}
  Johnson's theorem~\cite[Theorem~7.2]{Johnson.AMNM} is more general, as it applies to a class of Banach $*$-algebras
  which does not include just C*-algebras.  However, in this context the constant $K$ depends on the constant of
  amenability of $A$, as it depends on the best possible norm of an approximate diagonal in $A\hat\otimes A$. Since
  every C*-algebra is $1$-amenable (see, for example, \cite{Runde.LA}), our version follows.
\end{remark}

The proof of Theorem \ref{thm:limits} relies heavily on the fact that the range algebra, being a von Neumann algebra, is a dual Banach algebra. The assumption of nuclearity for elements of the class $\mathcal C$ is crucial due to the application of \cite[Theorem 3.1]{Johnson.AMNM}.

\begin{proof}[Proof of Theorem \ref{thm:limits}]
Let $\epsilon>0$, $A\in\mathcal D_{\mathcal C}$, $M\in\mathcal M$ and fix $\{A_\lambda\}_{\lambda\in\Lambda}$ be a directed system of algebras in $\SC$ with direct limit $A$. Fix a nonprincipal ultrafilter $\mathcal U$ on $\Lambda$ and let $\eta=\frac{\epsilon^2}{K^2}$ where $K$ is given by Theorem \ref{thm:johnrevised}.

As $(\mathcal C,\mathcal M)$ is Ulam stable by hypothesis, we can fix $\delta$ such that whenever $C\in\SC$,
$M\in\mathcal M$, and $\phi\colon C\to M$ is a $\delta$-$^*$-homomorphism, there is a $^*$-homomorphism $\psi$ with
$\norm{\psi-\phi}<\eta$. Let $\rho\colon A\to M$ be a $\delta$-$^*$-homomorphism. Then, for every $\lambda\in\Lambda$, there is a $^*$-homomorphism $\psi_\lambda$ with $\psi_\lambda\colon A_\lambda\to M$ such that 
\[
\norm{\psi_\lambda-\rho\restriction A_\lambda} < \eta
\] 
We extend each map $\psi_\lambda$ to $\bigcup_{\mu\in\Lambda} A_\mu$, setting $\psi_\lambda(a)=0$ if $a\notin A_\lambda$. Note that for every $a\in \bigcup A_\lambda$ there is $\lambda_0$ such that for all $\lambda\geq \lambda_0$ we have that $\norm{\psi_\lambda(a)-\rho(a)}<\eta$. For every $a\in\bigcup A_\lambda$, define 
\[
\psi(a)=\textrm{WOT}-\lim_{\mathcal U}\psi_\lambda(a)\in M.
\]
 Such a limit exists, since $M$ is a von Neumann algebra and $\norm{\psi_\lambda(a)}\leq\norm{a}$ for every
 $\lambda\in\Lambda$. In particular the map $\psi$ is a continuous, bounded, unital, linear map with domain equal to
 $\bigcup A_\lambda$, so it can be extended to a linear (actually, a completely positive and contractive) map 
\[
\tilde\psi\colon A\to M.
\]
By Lemma~\ref{lem:closeness}, for every $a\in \bigcup A_\lambda$ with $\norm{a} \le 1$, we have $\norm{\tilde{\psi}(a) -
\rho(a)} \le \eta$.  It follows that $\tilde{\psi}$ is $4\eta$-multiplicative.

As $M$ is a von Neumann algebra, and particular a dual Banach algebra, we can now apply \cite[Theorem 7.2]{Johnson.AMNM}
in the version of Theorem \ref{thm:johnrevised} to get a $^*$-homomorphism $\psi'\colon A\to M$ with
\[\norm{\rho-\psi'}<16K\eta^{1/2}=16\epsilon.\] The conclusion follows.
\end{proof}

As promised, we conclude with a connection to the perturbation theory of C*-algebras.  Recall that the Kadison-Kastler distance between two subalgebras of $\mathcal B(H)$ is given by 
\[
  d_{KK}(A,B)=\max\{\sup_{x\in A,\norm{x}=1}\inf_{y\in B,\norm{y}=1}\norm{x-y}, \sup_{x\in B,\norm{x}=1}\inf_{y\in A,\norm{y}=1}\norm{x-y}\}.
\]
Phillips and Raeburn proved, in~\cite{PhR.AF}, that any two AF subalgebras of $\SB(H)$ which are sufficiently close in
the Kadison-Kastler metric must be unitarily conjugate.  However, they were not able to control the distance between the
unitary and the identity operator in terms of the distance between the two AF algebras.  This problem remains open, although
partial results have been obtained in~\cite{CSSWW} (see Theorem~4.3 there).  Below, we give an equivalent condition in terms of approximate maps.
 \begin{corollary}\label{cor:ulamvsperturb}
The following conditions are equivalent:
\begin{enumerate}
\item\label{cond1} For every $\e>0$ there is $\delta>0$ such that whenever $A\in\mathcal{AF}$ and $\phi\colon A\to A$ is a $\delta$-$^*$-isomorphism there is a $^*$-isomorphism $\psi\colon A\to A$ with $\norm{\phi-\psi}<\e$;
\item\label{cond2} For every $\e>0$ there is $\delta>0$ such that whenever $A\in\mathcal{AF}$,  and $A_1,A_2\subseteq \mathcal B(H)$ are
isomorphic copies of $A$ with $d_{KK}(A_1,A_2) < \delta$, there is a $^*$-isomorphism $\phi\colon A_1\to A_2$ with
$\norm{\phi-\mathrm{id}}<\epsilon$.
\end{enumerate}
\end{corollary}
\begin{proof}
For \eqref{cond1} $\implies$ \eqref{cond2}, let $\e>0$ and use condition~\eqref{cond1} with $\e/2$ to choose a $\delta$, with $\delta
\le 2\e$.  Fix an AF algebra $A$ and two isomorphic copies $A_1,A_2\subseteq \mathcal B(H)$ as in
condition~\eqref{cond2}, with $d_{KK}(A_1,A_2)<\delta/4$. Define a map $\phi\colon A_1\to A_2$ assigning to every $x\in
A_1$ some $y\in A_2$ with $\norm{x} = \norm{y}$ and $\norm{x-y}\le\norm{x} \delta/4$. It is routine to verify that
$\phi$ is a $\delta$-$^*$-isomorphism, so, thanks to condition \eqref{cond1} and our choice of $\delta$ it can be
perturbed to a $^*$-isomorphism $\psi\colon A_1\to A_2$ that is $\e/2$-close to $\phi$.  Since $\phi$ is $\e/2$-close to
the identity, we have verified condition \eqref{cond2}.

For \eqref{cond2} $\implies$ \eqref{cond1}, let again $\e>0$ and use condition~\eqref{cond2} to choose $\delta>0$ such that for every
pair of unital, AF subalgebras of $\SB(H)$ which are $\delta$-close in the Kadison-Kastler metric, there is a
$^*$-isomorphism between them that differs from the identity by less than $\e/2$.  We assume that
$\delta \le \e/2$.  We also use Corollary~\ref{cor:AF} to choose $\eta>0$, with $\eta<\delta$, such that for every
unital AF algebra $A$ and every $\eta$-$^*$-homomorphism $\sigma : A\to \SB(H)$, there is a $^*$-homomorphism $\tau :
A\to\SB(H)$ with $\norm{\sigma - \tau} < \delta$.  

Now fix a unital AF algebra $A$ and an $\eta$-$^*$-isomorphism $\phi\colon
A\to A$.  Let $\pi : A\to \SB(H)$ be a faithful representation of $A$.  By factoring through $\pi$, we may define an
$\eta$-$^*$-isomorphism $\phi' : \pi(A)\to \pi(A)$ with $\pi \circ \phi = \phi' \circ \pi$.  By our choice of $\eta$,
there is a $^*$-homomorphism $\psi' : \pi(A)\to \SB(H)$ with $\norm{\phi' - \psi'} < \delta$.  In particular, we have
that $d_{KK}(\pi(A), \psi'(\pi(A))) < \delta$, so by our choice of $\delta$,
there is a $^*$-isomorphism $\sigma : \psi'(\pi(A))\to \pi(A)$ with $\norm{\sigma - \id} <
\e/2$.  Then the map $\psi = \pi^{-1}\circ \sigma\circ \psi'\circ \pi$ is a $^*$-isomorphism from $A$ to $A$ with
$\norm{\phi - \psi} < \e$, as required.

\end{proof}
\bibliographystyle{amsalpha}
\bibliography{library}
\end{document}